\documentclass[11pt]{article}
\usepackage[centertags]{amsmath}
\usepackage{amssymb}
\usepackage{amsthm}
\usepackage{makeidx}
\usepackage{color}
\usepackage{url}
\newtheorem{thm}{Theorem}[section]
\newtheorem{cor}[thm]{Corollary}

\newtheorem{prop}[thm]{Proposition}
\theoremstyle{definition}
\newtheorem{defn}{Definition}[section]
\theoremstyle{definition}

\newtheorem{rem}{Remark}[section]
\numberwithin{equation}{section}

\newcommand{\punt}{\boldsymbol{.}}
\newcommand{\K}{\mathfrak{K}}
\newcommand{\LL}{\mathfrak{L}}
\newcommand{\kcon}{{\scriptscriptstyle(\gamma)}}
\newcommand{\ab}{\bar{\alpha}_{\scriptscriptstyle D}^{\scriptscriptstyle <-1>}}
\newcommand{\dip}{{{{}_{\scriptscriptstyle D}}^{\scriptscriptstyle <-1>}}_{\scriptscriptstyle P}}
\newcommand{\free}{\boxplus}
\newcommand{\park}{\mathrm{park}}
\begin{document}
\title{Cumulants and convolutions via Abel polynomials}
\author{E. Di Nardo, P. Petrullo, D. Senato}
\date{\today}
\maketitle
\begin{center}
\textsf{Dipartimento di Matematica e Informatica, Universit\`a
degli Studi della Basilicata, via dell'Ateneo Lucano
10, 85100 Potenza, Italia}.\\
\goodbreak
\small\verb"elvira.dinardo@unibas.it, p.petrullo@gmail.com,"\\
\verb"domenico.senato@unibas.it"
\end{center}
\begin{abstract}
We provide an unifying polynomial expression giving moments in terms
of cumulants, and viceversa, holding in the classical, boolean and free
setting. This is done by using a symbolic treatment of Abel
polynomials. As a by-product, we show that in the free cumulant theory
the volume polynomial of Pitman and Stanley plays the role of
the complete Bell exponential polynomial in the classical theory. Moreover
via generalized Abel polynomials we construct a new class of
cumulants, including the classical, boolean and free ones, and the
convolutions linearized by them. Finally, via an umbral Fourier
transform, we state a explicit connection between boolean and
free convolution.
\end{abstract}
\textsf{\textbf{keywords}: umbral calculus, free cumulant, boolean
cumulant, classical cumulant, volume polynomial, Abel polynomials.}\\\\
\textsf{\textsf{AMS subject classification}: 60C05, 06A11, 05A40}\\
%
%
\section{Introduction}
Cumulants linearize the convolution of probability measures in the
three universal probability theories: classical, boolean and free.
The last is a noncommutative probability theory introduced by
Voiculescu~\cite{Voi3} with a view to tackle certain problems in
operator algebra. More precisely, a new kind of independence is
defined by replacing tensor products with free products and this
can help understand the Von Neumann algebras of free groups.
The combinatorics underlying this subject is based on the notion of
noncrossing partition, whose first systematical study is due to
Kreweras~\cite{Krew} and Poupard~\cite{Pou}. Within free
probability, noncrossing partitions are extensively used  by
Speicher~\cite{NicSpe1}. Speicher takes his lead from the
definition of classical multilinear cumulants in terms of the
M\"obius function. However, he changes the lattice where the
M\"obius inversion formula is applied. Instead of using the
lattice of all partitions of a finite set, he uses the smaller
lattice of noncrossing partitions.  Such a new family of cumulants,
known as free cumulants, turns out to be the semi-invariants of Voiculescu,
originally introduced via the $R$-transform. Biane~\cite{Bia} has shown
how free cumulants can be used to obtain asymptotical estimations of the
characters of large symmetric groups.

As is well known, some results of noncrossing partition theory can be
recovered via Lagrange inversion formula. Recently, a simple
expression of Lagrange inversion formula has been given by Di
Nardo and Senato~\cite{DiNaSen1} within classical umbral calculus.
This paper arises from this new Lagrange symbolic formula.

The classical umbral calculus~\cite{RotTay} is a renewed version
of the celebrated umbral calculus of Roman and
Rota~\cite{Roman-Rota}. It consists of a symbolic technique to
deal with sequences of numbers, indexed by nonnegative integers,
where the subscripts are treated as powers. Recently Di Nardo and
Senato~\cite{DiNaSen1, DiNaSen2} have developed this umbral
language in view of probabilistic applications. Moreover the
umbral syntax has been fruitfully used in computational
k-statistics and their generalizations~\cite{bernoulli}. The first
algebraic approach to this topic was given by McCullagh~\cite{McC}
and Speed~\cite{Speed}. Applications to bilinear generating
functions for polynomial sequences are given by Gessel~\cite{Ges}.

Rota and Shen~\cite{Rota-Shen} have already used umbral methods in
exploring some algebraic properties of cumulants, only in the
classical theory. They have proved that the umbral handling of
cumulants encodes and simplifies their combinatorics properties.
In this paper, we go further showing how the umbral syntax allows
us to explore the more hidden connection between the theory of
free cumulants and that of classical and boolean cumulants.

As pointed out in~\cite{NSTV}, the recent results of Belinschi and
Nica ~\cite{Belinschi} revealed a deeper connection between free
and boolean convolution that deserves a further clarification.
Indeed, this connection cannot be encoded in a straight way in the
formal power series language. We provide this connection via an
umbral Fourier transform. Moreover, quite surprisingly, the umbral
methods bring to the light that the key to manage all these
families of cumulants is the connection between binomial sequences
and Abel polynomials~\cite{RST}. This connection gives the chance
to find a new and very simple parametrization of free cumulants in
terms of moments. If $\alpha$ is the umbra representing the
moments and ${\mathfrak{K}}_{\scriptscriptstyle \alpha}$ is the
umbra representing the free cumulants, then
${\bar{\mathfrak{K}}}_{\scriptscriptstyle \alpha}^n \simeq
\bar{\alpha} ( \bar{\alpha} - n.\bar{\alpha})^{n-1}.$ This
parametrization closely parallels the one connecting cumulants and
moments, either in the classical or in the boolean setting, which
are respectively $\kappa_ {\scriptscriptstyle\alpha}^n
\simeq\alpha(\alpha-1\punt\alpha)^{n-1}$ and
$\bar{\eta}_{\scriptscriptstyle\alpha}^n \simeq
\bar{\alpha}(\bar{\alpha}-2 \punt \bar{\alpha})^{n-1},$ where
$\kappa_ {\scriptscriptstyle\alpha}^n$ denotes the $n$-th
classical cumulant and $\bar{\eta}_{\scriptscriptstyle\alpha}^n$
the $n$-th boolean cumulant.

The inverse expression giving moments in terms of free cumulants is
obtained (up to a sign) simply by swapping the umbra representing
moments with the umbra representing its free cumulants,
$\bar{\alpha}^n \simeq {\bar{\mathfrak{K}}}_{\scriptscriptstyle \alpha} (
{\bar{\mathfrak{K}}}_{\scriptscriptstyle \alpha} + n.
{\bar{\mathfrak{K}}}_{\scriptscriptstyle \alpha})^{n-1}.$  It is remarkable that the
polynomial, on the right side of the previous expression,
looks like the volume polynomials of Pitman and Stanley~\cite{PitStan} obtained when
the indeterminates are replaced by scalars, $V_n(a,a,\ldots,a)=a ( a + n \, a)^{n-1}.$ So
we prove that moments of an umbra can be recovered from volume
polynomials of Pitman and Stanley~\cite{PitStan} when the
indeterminates are replaced with the uncorrelated and similar free
cumulant umbrae. In other words, in the free cumulant theory the volume
polynomials are the analogs of the complete Bell exponential
polynomials in the classical cumulant theory.

The paper is structured as follows. In Section~2, we recall
the combinatorics of classical, boolean and free cumulants
with the aim to demonstrate how the umbral syntax provides an unifying
framework to deal with these number sequences. Indeed, in Section~3,
after recalling the umbral syntax, a theorem embedding the algebras of
multiplicative functions on the posets of all partitions and of
all interval partitions of a finite set in the classical umbral
calculus is proved. In this section we also recall the umbral
theory of classical cumulants  and we show how the umbral theory
of boolean cumulants is easily deduced from the classical one by
introducing the boolean unity umbra. A symbolic theory of free
cumulants closes the section. We also show that
Catalan numbers are the moments of the unique umbra whose free
cumulants are all equal to 1. In Section 4, we state the
connection  between volume polynomials and free cumulants.
In the last section we introduce a new class
of cumulants, including the classical, boolean
and free ones, and the convolutions linearized by them.
%
\section{Cumulants and convolutions}
The combinatorics of classical, free and boolean cumulants were
studied by Lehner \cite{Lehn1,Lehn2}, Speicher~\cite{Spe2}, and
Speicher and Wouroudi~\cite{SpeWou}. In the following we recall
the main results of their approach.

Denote by $[n]$ the set of positive integers $\{1,2,\ldots,n\}$
and by $\Pi_n$ the set of all partitions of $[n]$. The algebra of
the multiplicative functions on the poset $(\Pi_n,\leq)$ (see~\cite{DRS} or~\cite{Stan3}), where $\leq$ is the \emph{refinement}
order, provides nice formulae when the coefficients of the
exponential formal power series $f[g(t)-1]$ are expressed in terms
of the coefficients of $f(t)$ and $g(t)$, where
$f(t)=1+\sum_{n\geq 1}f_n\frac{t^n}{n!},$ and $g(t)=1+\sum_{n
\geq 1}g_n\frac{t^n}{n!}.$
For example, the coefficients of $\log f(t),$ expanded in an
exponential power series, are known as \emph{formal cumulants} of
$f(t).$ When $f(t)$ is the moment generating function of a random
variable $X,$ this sequence has a meaning which is not purely
formal. For example, cumulants of order $2,3$ and $4$ concur in
characterizing the variance, the skewness and the curtosis of a
random variable.
\par
Let us recall some well known facts on
multiplicative functions. We denote the minimum and the maximum of
the poset $\Pi_n$ by $\mathbf{0}_n$ and $\mathbf{1}_n$
respectively. The number of blocks of a given $\pi\in\Pi_n$ will be denoted by $\ell(\pi)$.

If $\sigma,\pi\in\Pi_n$ and
$[\sigma,\pi]=\{\tau\in\Pi_n\,|\,\sigma\leq\tau\leq\pi\}$, then
there is an unique sequence of non-negative integers
$(k_1,k_2,\ldots,k_n)$ with $k_1 + 2 k_2+ \cdots + n k_n = \ell(\sigma)$
and $k_1 + k_2 + \cdots + k_n = \ell(\pi)$ such that
\begin{equation}\label{id:dec int}
[\sigma,\pi]\cong\Pi_1^{k_1}\times\Pi_2^{k_2}\times\cdots\times\Pi_n^
{k_n}.
\end{equation}
In particular, if $\pi$ has exactly $m_i$ blocks of cardinality $i$, then
\begin{equation}
[\mathbf{0}_n,\pi]\cong\Pi_1^{m_1}\times\Pi_2^{m_2}
\times\cdots\times\Pi_n^{m_n}.
\label{(0d)}
\end{equation}
Then $[\pi,\mathbf{1}_n]\cong\Pi_{\ell(\pi)}.$
The sequence $(k_1,k_2,\ldots,k_n)$ is called the \emph{type} of
the interval $[\sigma,\pi],$ where
\begin{equation}\label{def:k_i}
k_i=\text{number of blocks of $\pi$ that are the union of $i$ blocks
of $\sigma$}.
\end{equation}
The vector $sh(\pi)=(c_1,c_2,\ldots,c_l)$ whose entries are the cardinalities
of the blocks of $\pi$, arranged in nondecreasing order, will be
called the \textit{shape} of $\pi$.
\par
A function ${\mathfrak f}:\Pi_n\times\Pi_n\rightarrow\mathbb{C}$ is said to
be \emph{multiplicative} if ${\mathfrak f}(\sigma,\pi)=f_1^{k_1}f_2^{k_2}\cdots f_n^{k_n},$
whenever (\ref{id:dec int}) holds and
$f_n := {\mathfrak f}(\mathbf{0}_n,\mathbf{1} _n)$. The
M\"obius function ${\mathfrak{\mu}}$, the Zeta function ${\mathfrak
\zeta}$ and the
Delta function ${\mathfrak \delta}$ are multiplicative functions with
$\mu_n=(-1)^{n-1}(n-1)!$, $\zeta_n=1$, and $\delta_n=\delta_{1,n}$
(the Kronecker delta).
\par
A convolution $\star$ is defined between two multiplicative
functions ${\mathfrak f}$ and ${\mathfrak g}$. We have
\begin{equation}\label{def:conv}
({\mathfrak f} \star {\mathfrak g})(\sigma,\pi)
:= \sum_{\sigma\leq\tau\leq\pi} {\mathfrak f}(\sigma,\tau) \, {\mathfrak g}
(\tau,\pi).
\end{equation}
The function ${\mathfrak f} \star {\mathfrak g}$ is also multiplicative. In
particular, if
${\mathfrak h}={\mathfrak f}\star {\mathfrak g}$, then $h_n=({\mathfrak f}\star
{\mathfrak g})(\mathbf{0}_n, \mathbf{1}_n)$ so
that
\begin{equation}\label{id:h=f star g}
h_n=\sum_{\tau\in\Pi_n}f_{\tau}\,g_{\ell(\tau)},
\end{equation}
where $f_{\tau}:= f_1^{m_1}f_2^{m_2}\cdots f_n^{m_n}$ and $m_i$ is
the number of blocks of $\tau$ of cardinality $i$. The function
$\delta$ is the identity with respect to the convolution $\star$.
Furthermore, $\mu$ and $\zeta$ are inverse each other with respect
to $\star$, that is $\mu\star\zeta=\zeta\star\mu=\delta$.
\begin{thm}\label{thm1}
Let ${\mathfrak f}$, ${\mathfrak g}$ and ${\mathfrak h}$ be three
multiplicative functions on the
lattice $(\Pi_n,\leq)$ with $f_n={\mathfrak f}(\mathbf{0}_n,\mathbf{1}_n)$,
$g_n={\mathfrak g}(\mathbf{0}_n,\mathbf{1}_n)$ and
$h_n={\mathfrak h}(\mathbf{0}_n,\mathbf{1}_n)$. If $f(t) =
1+\sum_{n\geq 1}f_n\frac{t^n}{n!},$ $g(t) = 1+\sum_{n\geq
1}g_n\frac{t^n}{n!}$ and $h(t) = 1+\sum_{n\geq
1}h_n\frac{t^n}{n!},$ then
\begin{equation}\label{id:mf vs star}
h(t)=f[g(t)-1]\Longleftrightarrow {\mathfrak h}={\mathfrak g}\star {\mathfrak
f}.
\end{equation}
\end{thm}
The formulae expressing cumulants $c_n$ in terms of moments
$m_n$, and viceversa,  are easily recovered from (\ref{id:h=f star g}).
Indeed, let $F(t)=1+\sum_{n\geq 1}m_n\frac{t^n}{n!}$  and $C(t)
= \log F(t) = 1+\sum_{n\geq 1} c_n\frac{t^n}{n!}$  its cumulant generating
function.
If ${\mathfrak m}$ and ${\mathfrak k}$ denote two multiplicative functions on
$(\Pi_n,\leq)$ such that ${\mathfrak m}(\mathbf{0}_n,\mathbf{1}_n)=m_n$ and ${\mathfrak k}
(\mathbf{0}_n,\mathbf{1}_n)=c_n$, then from Theorem~\ref{thm1} we have
\begin{equation}\label{id:M vs C}
{\setlength\arraycolsep{2pt}\left\{\begin{array}{lll} {\mathfrak
k}&=& {\mathfrak m} \star\mu,
\\ {\mathfrak m}& = & {\mathfrak k} \star\zeta.
\end{array}\right.}
\end{equation}
Free cumulants occur in noncommutative context of
probability theory (see for instance~\cite{Voi1}). A
\textit{noncommutative probability space} is a pair
$(\mathcal{A},\varphi)$, where $\mathcal{A}$ is a unital
noncommutative algebra and
$\varphi:\mathcal{A}\longrightarrow\mathbb{C}$ is a unital linear
functional. An element $X$ of $\mathcal{A}$ is called
\textit{noncommutative  random variable}. The $n$-th
\textit{moment} of $X$ is the complex number $m_n=\varphi(X^n)$,
the \textit{distribution} of $X$ is the collection of its moments
$(\varphi(X),\varphi(X^2),\varphi(X^3),\ldots)$. The moment
generating function of $X$ is the formal power series
\begin{equation}\label{def:free mom gen fun}
M(t)=1+\sum_{n\geq 1}m_n t^n.
\end{equation}
The \textit{noncrossing  (or \textit{free}) cumulants} of
$X$ are the coefficients $r_n$ of the ordinary power series $R(t)
=1+r_1t+r_2t^2+\ldots$ such that
\begin{equation}\label{def:free cum gen fun}
M(t)=R[tM(t)].
\end{equation}
This relation between cumulants and moments of a noncommutative
random variable has been found by Speicher~\cite{Spe1} and
characterize the free cumulants introduced by
Voiculescu~\cite{Voi1}, so we assume \eqref{def:free cum gen fun}
as a definition of free cumulants (see also~\cite{NicSpe}
and~\cite{Spe2}). Moreover, Speicher~\cite{Spe1} has shown that an
identity analogous to (\ref{id:M vs C}) holds between free
cumulants $\{r_n\}_{n\geq 1}$ and moments $\{m_n\}_{n\geq 1}$ of a
(noncommutative) random variable $X$, if we change the lattice of
partitions of a set into the lattice of noncrossing partitions
$(\mathcal{NC}_n,\leq)$.
\par
A \textit{noncrossing partition} $\pi=\{B_1,B_2,\ldots,B_s\}$ of
the set $[n]$ is a partition such that if $1\leq h<l<k<m\leq n$,
with $h,k\in B_j$, and $l,m\in B_{j\,'}$, then $j=j\,'$ (see~\cite{Krew},~\cite{Sim} and~\cite{Pou} for a detailed handling).
Let $\mathcal{NC}_n$ denote the set of all noncrossing
partitions of $[n]$. Its cardinality is equal to
the $n$-th Catalan number $\mathcal{C}_n.$
The convolution $\ast$ defined on
the multiplicative functions on the the lattice $(\mathcal{NC}_n,\leq)$ is given by
\begin{equation}\label{def:conv'}
({\mathfrak f} \ast {\mathfrak g})(\sigma,\pi)
:= \sum_{\scriptscriptstyle{\sigma\leq\tau\leq\pi \atop
\tau\in\mathcal{NC}_n}}{\mathfrak f}(\sigma,\tau) \, {\mathfrak g}(\tau,\pi).
\end{equation}
Following Nica and Speicher~\cite{NicSpe}, if $\tilde{\tau}$ is the Kreweras complement of a noncrossing partition $\tau$, then ${\mathfrak f}(\mathbf{0}_n,\tau)=\mathfrak{f}_{\tau}$ and ${\mathfrak g}(\tau,\mathbf{1}_n)=\mathfrak{g}_{\tilde{\tau}}$. Hence,  if ${\mathfrak h}={\mathfrak f} \ast {\mathfrak g}$ then
\begin{equation}\label{id:h=f ast g}
h_n=\sum_{\tau\in\mathcal{NC}_n}f_{\tau}\, {\mathfrak g}_{\tilde{\tau}}.
\end{equation}
If we denote by $\zeta_{\scriptscriptstyle\mathcal{NC}}$ and $\mu_{\scriptscriptstyle\mathcal{NC}}$ the Zeta
function and the M\"obius function on the noncrossing partition lattice respectively, then we have ${\mathfrak h} = {\mathfrak f} \ast
\zeta_{\scriptscriptstyle\mathcal{NC}}$ if and only if ${\mathfrak f} =
{\mathfrak h} \ast\mu_{\scriptscriptstyle\mathcal{NC}}$.
\begin{thm}[Speicher~\cite{Spe1}]\label{thm2}
Let $X$ be a noncommutative random variable with moment generating
function $M(t)$ and free cumulant generating function $R(t)$ as in
(\ref{def:free mom gen fun}) and (\ref{def:free cum gen fun}). If
${\mathfrak m}$ and ${\mathfrak r}$ are two multiplicative functions on the
lattice
$(\mathcal{NC}_n,\leq)$ such that
${\mathfrak m}(\mathbf{0}_n,\mathbf{1}_n)=m_n$ and
${\mathfrak r}(\mathbf{0}_n,\mathbf{1}_n)=r_n$, then
\begin{displaymath}
{\setlength\arraycolsep{2pt}\left\{\begin{array}{lll}{\mathfrak
r}& = &{\mathfrak m} \ast\mu_
{\scriptscriptstyle\mathcal{NC}},\\
{\mathfrak m} &= &{\mathfrak r}
\ast\zeta_{\scriptscriptstyle\mathcal{NC}}.\end {array}\right.}
\end{displaymath}
\end{thm}
The notion of \textit{boolean cumulants} arises from considering
the \textit{boolean convolution} of probability measures~\cite{SpeWou}. Within stochastic differential
equations, this family of cumulants is also known as \lq\lq
partial cumulants\rq\rq.  The boolean cumulants of $X$ are the coefficients $h_n$
of the ordinary delta series $H(t)=h_1t+h_2t^2+\ldots$ such that
\begin{equation}\label{def:bool cum}
M(t)=\frac{1}{1-H(t)},
\end{equation}
where $M(t)$ is the same as (\ref{def:free mom
gen fun}). From a combinatorial point of view, the formulae involving moments
and boolean cumulants are recovered by defining a convolution
$\diamond$ on the multiplicative functions on the lattice of interval partitions (see~\cite{VW}).
This lattice turns out to be isomorphic to the boolean lattice of a
$n$-set, from which the name of boolean convolution has been derived. A partition
$\pi$ of $\Pi_n$ is said to be an \textit{interval partition} if
each block $B_i$ of $\pi$ is an interval $[a_i,b_i]$ of $[n]$,
that is $B_i=[a_i,b_i]=\{x\in [n]\,|\,a_i\leq x\leq b_i\},$
where $a_i,b_i\in [n]$. We denote by $\mathcal{I}_n$ the subset of
$\Pi_n$ of all the interval partitions. The pair
$(\mathcal{I}_n,\leq)$, where $\leq$ is the refinement order, is a
lattice. The type of each interval $[\sigma,\tau]$ in
$\mathcal{I}_n$ is the same as (\ref{def:k_i}). Let ${\mathfrak f}$ and
${\mathfrak g}$ be two multiplicative functions on the interval
partition lattice. We define the convolution ${\mathfrak h} =
{\mathfrak f} \diamond {\mathfrak g}$, which is also
multiplicative, by
\begin{equation}\label{def:conv''}
({\mathfrak f} \diamond {\mathfrak g})(\sigma,\pi)
:= \sum_{\scriptscriptstyle{\sigma\leq\tau\leq\pi \atop
\tau\in\mathcal{I}_n}} {\mathfrak f}(\sigma,\tau) {\mathfrak g}(\tau,\pi).
\end{equation}
So if ${\mathfrak h} = {\mathfrak f}\diamond {\mathfrak g},$ then
\begin{equation}\label{id:h=f diamond g}
h_n=\sum_{\tau\in\mathcal{I}_n}f_{\tau}\, g_{\ell(\tau)}.
\end{equation}
Given the power series $H(t)$ and $M(t)$ in (\ref{def:bool cum}),
if ${\mathfrak h}$ and ${\mathfrak m}$ are two multiplicative
functions on the lattice of interval partitions, with ${\mathfrak
m}(\mathbf{0}_n,\mathbf{1}_n)=m_n$ and ${\mathfrak
h}(\mathbf{0}_n,\mathbf{1}_n)=h_n$, then we have
\begin{equation}\label{id:M vs H}
{\setlength\arraycolsep{2pt}\left\{\begin{array}{lll}{\mathfrak
h}& =& {\mathfrak m} \diamond\mu_
{\scriptscriptstyle\mathcal{I}},\\
{\mathfrak m} &=& {\mathfrak h}
\diamond\zeta_{\scriptscriptstyle\mathcal{I}},\end {array}\right.}
\end{equation}
where $\mu_{\scriptscriptstyle\mathcal{I}}$ and
$\zeta_{\scriptscriptstyle\mathcal{I}}$ are the M\"obius function
and the zeta function on $(\mathcal{I}_n,\leq)$.
\begin{thm}\label{thm3}
Let ${\mathfrak f}$, ${\mathfrak g}$ and ${\mathfrak h}$ be three
multiplicative functions on the lattice $(\mathcal{I}_n,\leq)$ with
$f_n={\mathfrak f}(\mathbf{0}_n,\mathbf{1}_n)$,
$g_n={\mathfrak g}(\mathbf{0}_n,\mathbf{1}_n)$ and
$h_n={\mathfrak h}(\mathbf{0}_n,\mathbf{1}_n)$. If $f(t) = 1+\sum_{n\geq 1}f_n t^n,$
$g(t)= 1+\sum_{n\geq 1}g_n t^n$ and $h(t)= 1+\sum_{n\geq 1}h_n t^n,$ then
\begin{equation}\label{id:mf vs gf3}
h(t)=f[g(t)-1]\Longleftrightarrow {\mathfrak h} = {\mathfrak g} \diamond
{\mathfrak f}.
\end{equation}
\end{thm}
Theorems~\ref{thm1} and \ref{thm3} state that the convolutions
$\star$ and $\diamond$ express the composition of exponential
power series and ordinary power series respectively. So these
convolutions are noncommutative. This is not true for the convolution \eqref{def:conv'}. In fact, the map $\tau\rightarrow\tilde{\tau}$ is an order-reversing bijection such that $sh(\tilde{\tilde{\tau}})=sh(\tau)$, and by virtue of \eqref{id:h=f ast g} we obtain $\mathfrak{f}\ast \mathfrak{g}=\mathfrak{g} \ast \mathfrak{f}$ (see~\cite{NicSpe} for more details).
%
\section{Symbolic methods for classical, boolean and free cumulants}
\indent
We start this section recalling the necessary tools of the umbral syntax;
main references are~\cite{DiNaSen1,DiNaSen2,DiNaSen3} and~\cite{RotTay}.
\par
Let us denote by $X$ the set of variables $X = \{x_1, x_2, \ldots, x_n \}.$
A classical umbral calculus consists of the following data: a set $A=\{\alpha,\beta,\ldots\}$,
called the \textit{alphabet}, whose elements are named \textit{umbrae};
a linear functional $E$, called the \textit{evaluation}, defined on the
polynomial ring $R[X][A]$ and taking value in $R[X]$ (a ring whose
quotient field is of characteristic zero), such that $E[1]=1$ and
$E[x_1^s x_2^m \cdots x_n^t \alpha^i\beta^j\cdots\gamma^k]=x_1^s x_2^m \cdots x_n^t E[\alpha^i]E[\beta^j]\cdots
E[\gamma^k]$ (\textit{uncorrelation} property) for all nonnegative integers $s,m, t,i,j,k$;
two special umbrae $\varepsilon$ (\textit{augmentation}) and
$u$ (\textit{unity}) such that $E[\varepsilon^i]=\delta_{0,i},$
and $E[u^i]=1,$ for $i=0,1,2,\ldots.$
\par
A sequence $(1,a_1,a_2,\ldots)$ of elements of $R$ is
\textit{represented} by a scalar umbra $\alpha$ if $E[\alpha^i]=a_i,$ for
$i=0,1,2,\ldots$. In this case we say that $a_i$ is the $i$-th
moment of $\alpha$. In the following the powers of an
umbra $\alpha$ will be also called moments, if this does not give
rise to misunderstandings. A sequence $(1,p_1,p_2,\ldots)$ of elements of $R[X],$
such that $p_n$ is of degree $n$ for all $n,$ is
\textit{represented} by a polynomial umbra $\psi$ if $E[\psi^i]=p_i$ for
$i=0,1,2,\ldots$.
\par
The factorial moments of a scalar umbra $\alpha$ are the elements $a_{(n)}
\in R$ such that $a_{(0)}=1$ and $a_{(n)} = E[
(\alpha)_n]= E[\alpha (\alpha - 1) (\alpha - 2) \cdots (\alpha - n +1)]$ for all $n \geq 1.$
The polynomial $(\alpha)_n$ is an umbral polynomial. More general,
an umbral polynomial is a polynomial $p \in R[X][A].$
The support of $p$ is the set of all umbrae occurring in $p.$ If $p$ and
$q$ are two umbral polynomials, then $p$ and $q$ are uncorrelated
if and only if their supports are disjoint. Moreover the polynomials
$p$ and $q$ are \textit{umbrally equivalent} if and only if $E[p] = E[q],$
in symbols $p \simeq q.$

Two umbrae are \textit{similar}, in symbols $\alpha\equiv\gamma$,
if and only if $E[\alpha^n] = E[\gamma^n]$ for all $n$.
So, each sequence is represented by infinite many uncorrelated (i.e. distinct) umbrae.
In the following, we shall denote by $\alpha,\alpha^{\prime},\alpha^{\prime
\prime},\ldots$ a family of similar and uncorrelated umbrae. We extend the alphabet $A$ with
the so-called {\it auxiliary umbrae} obtained via operations among
similar umbrae. This leads to the construction of a {\it saturated
umbral calculus} in which auxiliary umbrae are treated as elements
of a suitable alphabet. For example, the symbol $n
\punt \alpha$ denotes an auxiliary umbra similar to the sum of $n$
distinct umbrae, each one similar to the umbra $\alpha,$ that is
$n \punt \alpha \equiv  \alpha^{\prime} +\alpha^{\prime \prime}+
\cdots + \alpha^ {\prime \prime \prime}.$ We remark that
$n\punt(\alpha+\gamma)\equiv n\punt\alpha+n\punt\gamma$, and
$\alpha\punt n\equiv n\alpha$, for every umbrae $\alpha$ and
$\gamma$ and for all nonnegative $n$.
\par
The \textit{generating function} $f(\alpha,t)$ of $\alpha$ is
$f(\alpha,t)= 1+\sum_{n\geq 1}a_n\frac{t^n}{n!}.$
A formal construction is given in~\cite{DiNaSen1}. In particular,
we have $f(\varepsilon,t)=1, f(u,t)=e^t$ and $f(n \punt \alpha,t) = f(\alpha,t)^n$.

Special umbrae are the \textit{Bell umbra} $\beta$ and the \textit{singleton umbra}
$\chi.$ The Bell umbra $\beta$ has moments
given by the Bell numbers so that $f(\beta,t)=e^{e^t-1}.$
The singleton umbra $\chi$ has moments $E[\chi^n]=1,$ if $n=0,1$,
and $E[\chi^n]=0$ otherwise, so that $f(\chi,t)=1+t.$
The derivative umbra $\alpha_{\scriptscriptstyle D}$ of an umbra
$\alpha$ is the umbra whose moments are $(\alpha_{\scriptscriptstyle D})^n
\simeq \partial_{\alpha} \alpha^n \simeq n \alpha^{n-1}$ for $n=1,2,\ldots$~\cite{DiNaSen3}. We have $f(\alpha_{\scriptscriptstyle D},t)=1 +t \, f(\alpha,t)$,
and in particular $E[\alpha_{\scriptscriptstyle D}]=1$.
\par
Given an umbra $\alpha$, the umbra denoted by $-1\punt\alpha$ is
uniquely determined (up to similarity) by the condition
$\alpha+(-1\punt\alpha)\equiv\varepsilon.$
The umbra $-1\punt\alpha$ is said to be the {\it inverse} of
$\alpha$. Its generating function is $f(\alpha,t)^{-1}.$ Then, the
umbra $-n \punt \alpha$ is the inverse of $n \punt \alpha$ and
$f(-n \punt \alpha) = f(\alpha,t)^{-n}.$
\par
A generalization of the auxiliary umbra
$n\punt\alpha$ (\textit{dot operation}) is introduced when $n$ is
replaced by an umbra $\gamma$. We denote by $\gamma\punt\alpha$ an auxiliary umbra with
moments
\begin{equation}\label{def:dot mom}
(\gamma\punt\alpha)^n\simeq\sum_{\lambda \vdash n}
\mathrm{d}_{\lambda}\,(\gamma)_{\ell(\lambda)}
\,(\alpha')^{\lambda_1}(\alpha'') ^{\lambda_2}
\cdots(\alpha''')^{\lambda_{\ell(\lambda)}},
\end{equation}
where the sum ranges over all the partitions $\lambda=
(\lambda_1,\lambda_2,\ldots, \lambda_l)$ of $n,$ where $l=\ell(\lambda)$ is the number of positive parts of $\lambda$, and
$\mathrm{d}_{\lambda}={n\choose \lambda}/[m(\lambda)_1!m(\lambda)_2! \cdots m(\lambda)_n!]$,
where $m(\lambda)_i$ denotes the number of parts of $\lambda$ equal to $i$.
From now on, we denote
$(\alpha')^{\lambda_1}(\alpha'')^{\lambda_2}
\cdots(\alpha''')^{\lambda_{\ell(\lambda)}}$ and $m(\lambda)_1!$
$m(\lambda)_2!$ $\cdots$ $m(\lambda)_n!$ by $\alpha_{\lambda}$ and
$m(\lambda)$! respectively. The generating function of
$\gamma\punt\alpha$ is $f(\gamma\punt\alpha,t)=f[\gamma,\log f(\alpha,t)].$
In particular, we have $\alpha\punt u\equiv u\punt\alpha\equiv\alpha$ for all $\alpha$ in
$A$, and $\chi\punt\beta\equiv\beta\punt\chi\equiv u$.
\par
The composition of $f(\gamma,t)$ and $f(\alpha,t)$ is the
generating function of $\gamma\punt\beta\punt\alpha,$
$f(\gamma\punt\beta\punt\alpha,t)=f[\gamma,f(\alpha,t)-1].$
The umbra $\gamma\punt\beta\punt\alpha$ is said to be the
\textit{composition umbra} of $\gamma$ and $\alpha$.
 The moments of $\gamma\punt\beta\punt\alpha$ are
\begin{equation}
(\gamma\punt\beta\punt\alpha)^n \simeq \sum_{\lambda \vdash n}
\mathrm{d}_{\lambda}\,\gamma^{\ell(\lambda)}\,
\alpha_{\lambda}.
\label{def:comp mom}
\end{equation}
In particular $(\gamma \punt \beta) \punt \alpha  \equiv \gamma \punt (\beta \punt
\alpha)$  and
\begin{equation}
\gamma_{\scriptscriptstyle D} \punt \beta \punt \alpha_
{\scriptscriptstyle D} \equiv (\alpha + \gamma \punt \beta \punt
\alpha_{\scriptscriptstyle D})_ {\scriptscriptstyle D}.
\label{(deriv)}
\end{equation}
Finally the symbol $\alpha^{\scriptscriptstyle<-1>}$ denotes an umbra
whose generating function is the compositional inverse $f^{\scriptscriptstyle<-1>}
(\alpha,t)$ of $f(\alpha,t).$ Such an umbra is uniquely
determined (up to similarity) by the relations
$\alpha\punt\beta\punt\alpha^{\scriptscriptstyle<-1>}\equiv\alpha^
{\scriptscriptstyle<-1>} \punt\beta\punt\alpha\equiv\chi.$
\begin{thm}\label{thm fund 1}
Let ${\mathfrak f}$, ${\mathfrak g}$ and ${\mathfrak h}$ be three
multiplicative functions on the
lattice $(\Pi_n,\leq)$. If $\alpha$, $\gamma$ and $\omega$ are
three umbrae with moments $\alpha^n\simeq f_n,$ $\gamma^n\simeq
g_n,$ and $\omega^n\simeq h_n,$ then we have
${\mathfrak h}={\mathfrak f}\star {\mathfrak g} \Longleftrightarrow
\omega\equiv\gamma\punt\beta\punt\alpha.$
\end{thm}
\begin{proof}
Note that in the equivalence (\ref{def:comp mom}), $\mathrm{d}_{\lambda}$ counts the number
of partitions of $\Pi_n$ of shape $\lambda$ so that
$(\gamma\punt\beta\punt\alpha)^n \simeq \sum_{\pi\in\Pi_n}\gamma^{\ell
(\pi)}\alpha_{\pi},$ where $\pi=\{B_1,B_2,\ldots,B_l\}\in\Pi_n$ and
we set $\alpha_{\pi}=(\alpha')^{|B_1|}(\alpha')^{|B_2|} \cdots(\alpha''')^{|B_l|}.$
The result follows by comparing this last equivalence
with (\ref{id:h=f star g}).
\end{proof}
\begin{rem} Theorem~\ref{thm fund 1} states that multiplicative functions
on ${(\Pi_n,\leq)}$ can be thought as umbrae, and the convolution $\star$ of two
multiplicative functions corresponds to a composition umbra. The umbra
$\chi\punt\chi$ is the umbral counterpart of the M\"obius function
$\mu$. In fact, $f(\chi\punt\chi,t)=1+\log(1+t)$ so that
$(\chi\punt\chi)^n \simeq (-1)^{n-1}(n-1)!=\mu_n.$
In addition, the umbral counterparts of the Zeta function
$\zeta$ and the Delta function $\delta$ are respectively the unity
umbra $u$ and the singleton umbra $\chi$. Hence the relations among
multiplicative functions can be interpreted in the umbral syntax.
For example, we have $\delta=
\mu\star\zeta = \zeta\star\mu$ similarly to
$\chi\equiv(\chi\punt\chi)\punt\beta\punt u\equiv
u\punt\beta\punt(\chi\punt\chi)$. Furthermore, an umbra $\alpha$ has a
compositional inverse $\alpha^{\scriptscriptstyle<-1>}$ if and only
if $E[\alpha] = a_1 \neq 0$. In analogy, a multiplicative function
${\mathfrak f}$ has an inverse respect to the convolution $\star$ if and only
if $f_1\neq 0$.
\end{rem}
\subsection{Classical cumulants}
Classical cumulants have been studied via the classical umbral calculus in~\cite{DiNaSen1}.
Here we state a new theorem concerning a parametrization of classical cumulants
and moments. This parametrization represents the trait d'union with the umbral theory
of boolean and free cumulants, that we introduce later on.

For each umbra $\alpha$, the $\alpha$-\textit{cumulant umbra} is an umbra, denoted by
$\kappa_{\scriptscriptstyle\alpha}$, similar to $\chi\punt\alpha$.
In particular we have $\kappa_{\scriptscriptstyle\alpha}
\equiv(\chi\punt\chi)\punt\beta\punt\alpha,$ and from Theorem~\ref{thm fund 1}, this similarity is the umbral version
of (\ref{id:M vs C}), if we assume $F(t)=f(\alpha,t)$ and
$C(t)=f(\kappa_{\scriptscriptstyle\alpha},t).$ The formulae expressing the cumulants
$\kappa_{\scriptscriptstyle\alpha}^n\simeq c_n$ in terms of their
moments $\alpha^n\simeq m_n$ are easily recovered from
(\ref{def:comp mom}):
\begin{equation}\label{id:cum vs mom}
c_n=\sum_{\lambda\vdash
n}\mathrm{d}_{\lambda}(-1)^{\ell(\lambda)-1}(\ell(\lambda)-1)! \, m_{\lambda},
\end{equation}
where $m_{\lambda}=E[\alpha_{\lambda}]$, so that
$m_{\lambda}=m_{\lambda_1}m_{\lambda_2}\cdots m_{\lambda_{\ell(\lambda)}}$.
The relation $\kappa_{\scriptscriptstyle\alpha}\equiv\chi\punt\alpha$
is inverted by $\alpha\equiv\beta\punt \kappa_{\scriptscriptstyle\alpha}$
by which we have $m_n=\sum_{\lambda\vdash n}\mathrm{d}_{\lambda}c_{\lambda}.$
In particular $m_n=Y_n(c_1,c_2,\ldots,c_n),$ where $Y_n$ is the complete Bell
exponential polynomial. The Bell umbra $\beta$ is the unique umbra, up to
similarity, having the sequence of cumulants $\{1\}_{n\geq 1}$, being
$\kappa_{\beta}\equiv\chi\punt\beta\equiv u$. Moreover, we have
$\beta^n\simeq\mathcal{B}_n=|\Pi_n|$.
\par
Compared with moments, cumulants are special sequences because of their properties
of additivity and homogeneity. The following theorem states these properties in
umbral terms. Recall that the \textit{disjoint sum} of the umbrae $\alpha$ and
$\gamma$ is  an auxiliary umbra such that $(\alpha \stackrel{\punt}{+} \gamma)^n\simeq\alpha^n+\gamma^n.$
\begin{thm}\label{thm:cum umb}
For all umbrae $\alpha,\gamma\in A$ and for all $c\in R,$ the following
properties hold:
$$\begin{array}{llll}
\kappa_{\scriptscriptstyle\alpha+\gamma} & \equiv &
\kappa_{\scriptscriptstyle\alpha} \stackrel{\punt}{+} \kappa_
{\scriptscriptstyle\gamma} \quad & \hbox{(\textit{additivity property})}; \\
\kappa_{\scriptscriptstyle\alpha+cu} & \equiv & \kappa_{\scriptscriptstyle\alpha} \stackrel{\punt}{+} c \chi,
& \hbox{(\textit{semi-invariance for translation property})}; \\
\kappa_{\scriptscriptstyle c\alpha} & \equiv & c
\kappa_{\scriptscriptstyle\alpha}.  & \hbox{(\textit{homogeneity property})}.
\end{array}$$
\end{thm}
\begin{thm} [Parametrization] \label{par1} If $\kappa_
{\scriptscriptstyle\alpha}$ is the $\alpha$-cumulant umbra, then
\begin{equation}
\alpha^n\simeq\kappa_{\scriptscriptstyle\alpha}(\kappa_
{\scriptscriptstyle\alpha}+ \alpha)^{n-1} \quad \hbox{and} \quad \kappa_
{\scriptscriptstyle\alpha}^n
\simeq\alpha(\alpha-1\punt\alpha)^{n-1}.\label{id:class param 1}
\end{equation}
\end{thm}
\begin{proof}
Sine for any umbra $\alpha \in A$ we have $(\beta \punt \alpha)^n
\simeq \alpha (\alpha + \beta \punt \alpha)^{n-1},$ see~\cite{DiNaSen2},
we obtain the former in equivalence (\ref{id:class param 1}) replacing
$\alpha$ by  $\kappa_{\scriptscriptstyle\alpha}\equiv\chi\punt\alpha.$
The latter can be proved as follows. We have
\begin{displaymath}
\alpha(\alpha-1\punt\alpha)^{n-1}\simeq\sum_{\scriptscriptstyle1\leq
k\leq n \atop \lambda\vdash n-k}{n-1\choose
k-1} \mathrm{d}_{\lambda}(-1)_{\ell(\lambda)} \, \alpha^k \, \alpha_{\lambda},
\end{displaymath}
and, setting $\lambda\leftarrow\lambda\cup k$ (i.e. a part equal to $k$ is joined with $\lambda$), we recover equation
(\ref{id:cum vs mom}).
\end{proof}
%
\subsection{Boolean cumulants}
%
The notion of boolean cumulant requires the connection
between umbrae and ordinary generating functions. We obtain this
connection simply by multiplying an umbra by the boolean unity
umbra $\bar{u},$ whose moments are $\bar{u}^n\simeq n!.$ In fact, if
$\alpha$ has moments $\alpha^n\simeq a_n$, the umbra $\bar{\alpha}
\equiv \bar{u}\alpha$ has generating function $f(\bar{\alpha},t)
=1+a_1t+a_2t^2+\cdots.$ Note that, $\alpha\equiv\gamma$ if and only if $\bar{\alpha}\equiv\bar{\gamma}$. The following theorem is the analogous
of Theorem~\ref{thm fund 1} for the lattice $(\mathcal{I}_n,\leq).$
\begin{thm}\label{thm fund 2}
Let ${\mathfrak f}$, ${\mathfrak g}$ and ${\mathfrak h}$ be three
multiplicative functions on the
lattice $(\mathcal{I}_n,\leq)$. If $\alpha$, $\gamma$ and $\omega$
are three umbrae with moments $\alpha^n\simeq f_n,$ $\gamma^n\simeq
g_n,$ and $\omega^n\simeq h_n,$ then we have
${\mathfrak h} = {\mathfrak f} \diamond {\mathfrak g} \Longleftrightarrow
\bar{\omega}\equiv\bar{\gamma}\punt\beta\punt\bar{\alpha}.$
\end{thm}
\begin{proof}
Since $\bar{\alpha}_{\lambda}\simeq\lambda!\alpha_{\lambda}$,
from (\ref{def:comp mom}) the
moments $h_n$ of $\bar{\gamma}\punt\beta\punt\bar{\alpha}$ are
\begin{equation}\label{id:bool 1}
h_n=\sum_{\lambda\vdash
n} \, \frac{\ell(\lambda)!}{m(\lambda)!} \, g_{\ell(\lambda)} \, f_{\lambda}.
\end{equation}
But $\ell(\lambda)!/m(\lambda)!$ is the number of interval
partitions of shape $\lambda$, so that ${\mathfrak h} = {\mathfrak
f} \diamond {\mathfrak g}.$
\end{proof}
Since the $\alpha$-cumulant umbra is such that $\kappa_{\alpha} \equiv
\chi. \alpha \equiv u^{\scriptscriptstyle<-1>} \punt \beta \punt \alpha,$
we define the $\alpha$-\textit{boolean cumulant umbra} by
taking the \lq\lq bar version\rq\rq of the previous similarity.
\begin{defn} \label{def:bool cum umb}
The $\alpha$-\textit{boolean cumulant umbra} is the umbra
$\eta_{\alpha}$ such that $\bar{\eta}_{\scriptscriptstyle\alpha}\equiv\bar{u}^{\scriptscriptstyle<-1>}
\punt\beta\punt\bar{\alpha}.$
\end{defn}
If $h_n$ denotes the $n$-th moment of $\alpha$-boolean cumulant umbra,
from (\ref{id:bool 1}) we have
\begin{equation}
h_n = \sum_{\lambda\vdash n}\,
\frac{\ell(\lambda)!}{m(\lambda)!}\, (-1)^{\ell(\lambda)-1} \,
m_{\lambda}. \label{id:bool vs mom}
\end{equation}
Definition  \ref{def:bool cum umb} is based on the following
proposition that states that $h_n$ in (\ref{id:bool vs mom})
are the same as the coefficients of $H(t)$ in (\ref{def:bool cum}).
\begin{prop}
If $\eta_{\scriptscriptstyle\alpha}$ is
the $\alpha$-\textit{boolean cumulant umbra}, then
$f(\bar{\eta}_{\scriptscriptstyle\alpha},t)= 2 - f(\bar{\alpha},t)^{-1}$
and $ f(\bar{\alpha},t) = (1 - [f(\bar{\eta}_{\scriptscriptstyle\alpha},t) -1])^{-1}.$
\end{prop}
\begin{proof}
We have $f(\bar{u},t)=(1-t)^{-1}$ so
$\bar{u} \equiv -1 \punt - \chi$ and $-1 \punt \bar{u} \equiv -
\chi.$  Moreover, we have $\bar{u}^{<-1>} \punt \beta \equiv -\chi \punt -\beta,$
since  $-\chi \punt - \beta \punt \bar{u} \equiv -\chi \punt -\beta \punt
-1 \punt -\chi \equiv -\chi \punt \beta \punt -1 \punt -1 \punt -
\chi \equiv -\chi \punt \beta \punt -\chi \equiv \chi,$ this because
$-\chi$ is the compositional inverse of itself and
$\bar{u}^{<-1>} \punt \beta \punt \bar{u} \equiv \chi.$
Therefore we have $f(\bar{u}^{<-1>} \punt \beta,t) = 2 - e^{-t},$
by which the results follow.
\end{proof}
\begin{thm}[Boolean Inversion Theorem] \label{id:bool inv}
If $\eta_{\scriptscriptstyle\alpha}$ is the $\alpha$-boolean
cumulant, then $\bar{\alpha}\equiv\bar{u}\punt\beta\punt
\bar{\eta}_{\scriptscriptstyle\alpha}.$
\end{thm}
\begin{proof}
The result follows from (\ref{def:bool cum umb}) by left dot product
of both sides with $\bar{u} \punt \beta$.
\end{proof}
The unique umbra (up to similarity) having sequence
of boolean cumulants $\{1\}_{n\geq 1}$ is an umbra $\alpha$ such
that $\bar{\alpha}\equiv\bar{u}\punt\beta\punt\bar{u}\equiv (2\bar{u})_
{\scriptscriptstyle D}.$ Since $\bar{\alpha}^n\simeq(2\bar{u})_{\scriptscriptstyle D}^n\simeq n(2\bar{u})^{n-1}\simeq n!2^{n-1}$, then such an umbra has moments $2^{n-1},$ that is the number
of interval partitions $\mathcal{I}_n$.
The following theorem gives a parametrization of boolean
cumulants and moments. The proof is omitted.
\begin{thm} [Boolean parametrization] \label{par2} If $\eta_{\scriptscriptstyle\alpha}
$ is the $\alpha$-boolean cumulant umbra, then
\begin{equation}
\bar{\alpha}^n\simeq \bar{\eta}_{\scriptscriptstyle\alpha}(\bar{\eta}_
{\scriptscriptstyle\alpha}+
2\punt\bar{\alpha})^{n-1} \quad \hbox{and} \quad
\bar{\eta}_{\scriptscriptstyle\alpha}^n\simeq\bar{\alpha}(\bar{\alpha}-2
\punt\bar{\alpha})^{n-1}.\label{id:bool param 1}
\end{equation}
\end{thm}
Similarly to the $\alpha$-cumulant umbra, we can state additivity and homogeneity properties
also for the $\alpha$-boolean cumulant umbra.
\begin{thm} [Homogeneity property] \label{id:hom boo}
If $\eta_{\scriptscriptstyle\alpha}$ is the $\alpha$-boolean
cumulant umbra, then $\eta_{\scriptscriptstyle c\alpha}\equiv c
\eta_{\scriptscriptstyle\alpha}.$
\end{thm}
\begin{proof}
Since $\bar{u}^{\scriptscriptstyle<-1>} \punt\beta\punt c \bar
{\alpha} \equiv c (\bar{u}^{\scriptscriptstyle<-1>} \punt\beta\punt \bar{\alpha})$ and $\overline{c\alpha} \equiv c \bar{\alpha}$, then from (\ref{def:bool cum umb}) we have $\bar{\eta}_{\scriptscriptstyle c\alpha}\equiv c\bar{\eta}_{\scriptscriptstyle\alpha}$ and finally $\eta_{\scriptscriptstyle c\alpha}\equiv c\eta_{\scriptscriptstyle\alpha}$.
\end{proof}
\begin{thm} [Additivity property]\label{thm:add bool}
If $\eta_{\scriptscriptstyle\alpha},
\eta_{\scriptscriptstyle\gamma} $ and
$\eta_{\scriptscriptstyle\xi}$ are the boolean cumulant umbrae of
$\alpha, \gamma$ and $\xi$ respectively, then
\begin{equation}\label{bool conv}
\eta_{\scriptscriptstyle\xi} \equiv \eta_
{\scriptscriptstyle\alpha} \stackrel{\punt}{+}
\eta_{\scriptscriptstyle\gamma} \Leftrightarrow -1 \punt
\bar{\xi} \equiv -1 \punt \bar{\alpha} \stackrel{\punt}{+} -1
\punt \bar{\gamma}.
\end{equation}
\end{thm}
\begin{proof}Let $-1 \punt \bar{\xi} \equiv -1 \punt \bar{\alpha} \stackrel{\punt}
{+} -1 \punt \bar{\gamma}$. Due to
$ -1 \punt \bar{\alpha} \equiv (- \chi \punt \beta) \punt \bar{\eta}_
{\scriptscriptstyle\alpha}$ , we have
$-\chi \punt \beta \punt \bar{\eta}_{\scriptscriptstyle\xi} \equiv -\chi
\punt \beta \punt \bar{\eta}_{\scriptscriptstyle\alpha}
\stackrel{\punt}{+} -\chi \punt \beta \punt
\bar{\eta}_{\scriptscriptstyle\gamma} \equiv -\chi \punt (\beta
\punt \bar{\eta}_{\scriptscriptstyle\alpha} + \beta \punt
\bar{\eta}_{\scriptscriptstyle\gamma})$ so that $\beta \punt
\bar{\eta}_{\scriptscriptstyle\xi} \equiv \beta \punt
\bar{\eta}_{\scriptscriptstyle\alpha} + \beta \punt
\bar{\eta}_{\scriptscriptstyle\gamma}$. Taking the left product of
both sides for $\chi,$ the result follows.
\end{proof}
We define the \textit{boolean convolution} of $\alpha$ and
$\gamma$ to be the umbra $\alpha\uplus\gamma$ such that
$\overline{\alpha\uplus\gamma}\equiv-1\punt(-1 \punt \bar{\alpha}
\stackrel{\punt}{+} -1 \punt \bar{\gamma}).$
Theorem~\ref{thm:add bool} assures this is the unique convolution
linearized by boolean cumulants. In this way, from (\ref{bool
conv}) we express the additivity property of the boolean cumulant
umbra with respect to the boolean convolution as follows
$\eta_{\scriptscriptstyle\alpha\uplus\gamma}\equiv
\eta_{\scriptscriptstyle\alpha}\stackrel{\punt}{+}\eta_{\scriptscriptstyle\gamma}.$
Since $\bar{\eta}_{cu} \equiv (\bar{u}^{\scriptscriptstyle<-1>}\punt\beta
\punt\bar{u})\punt c\equiv\chi\punt c\equiv c\chi,$ from (\ref{bool
conv}) we have
$\eta_{\scriptscriptstyle\alpha\uplus c u} \equiv \eta_
{\scriptscriptstyle\alpha} \stackrel{\punt}{+} c \, \chi$ that
gives the \textit{semi-invariance property}.
\par
Once more, note the analogy with the convolution linearized by classical cumulants,
that is $\alpha+\gamma\equiv -1\punt(-1\punt\alpha+-1\punt\gamma).$
%
\subsection{Free cumulants}
\begin{defn}[Free cumulant umbra] \label{free}
For a given umbra $\alpha,$ the unique umbra $\mathfrak{K}_
{\scriptscriptstyle \alpha}$ (up to similarity) such that
$(-1 \punt {\bar{\mathfrak{K}}}_{\scriptscriptstyle \alpha})_
{\scriptscriptstyle D} \equiv {\bar{\alpha}_{\scriptscriptstyle D}}^
{{\scriptscriptstyle <-1>}}$ is called the \textit{free
cumulant umbra} of $\alpha$.
\end{defn}
The moments of $\mathfrak{K}_{\scriptscriptstyle \alpha}$
will be called \textit{free cumulants} of the umbra $\alpha.$
Definition  \ref{free} is based on the following
proposition that states that the free cumulants of an
umbra $\alpha,$ whose moments are $m_n,$ are the coefficients
of $R(t)$ in (\ref{def:free cum gen fun}).
\begin{prop} \label{def2} If ${\mathfrak{K}}_{\scriptscriptstyle \alpha}$ is the
free cumulant umbra of $\alpha$, then $\bar{\alpha} \equiv
{\bar{\mathfrak{K}}}_{\scriptscriptstyle \alpha} \punt
\beta \punt \bar{\alpha}_{\scriptscriptstyle D}.$
\end{prop}
\begin{proof} By using Definition~\ref{free}, we have
$\bar{\alpha}_{\scriptscriptstyle D}  \punt \beta \punt \bar{\alpha}_
{\scriptscriptstyle D}^{{\scriptscriptstyle <-1>}} \equiv
\bar{\alpha}_{\scriptscriptstyle D}  \punt \beta \punt (-1 \punt
\bar{\mathfrak{K}}_{\scriptscriptstyle \alpha})_{\scriptscriptstyle D}
$ and via (\ref{(deriv)}) we obtain
$\bar{\alpha}_{\scriptscriptstyle D}  \punt \beta \punt (-1 \punt \bar{\mathfrak
{K}}_{\scriptscriptstyle \alpha})_{\scriptscriptstyle D} \equiv
(\bar{\alpha} - 1 \punt \bar{\mathfrak{K}}_{\scriptscriptstyle \alpha} \punt
\beta \punt \bar{\alpha}_{\scriptscriptstyle D})_{\scriptscriptstyle
D}.$ As $\bar{\alpha}_{\scriptscriptstyle D}  \punt \beta \punt
{\bar{\alpha}_{\scriptscriptstyle D}^{\scriptscriptstyle<-1>}} \equiv
\chi$, then $\bar{\alpha} - 1 \punt \bar{\mathfrak{K}}_{\scriptscriptstyle \alpha} \punt
\beta \punt \bar{\alpha}_{\scriptscriptstyle D} \equiv \varepsilon
\Leftrightarrow \bar{\alpha} \equiv \bar{\mathfrak{K}}_{\scriptscriptstyle
\alpha} \punt \beta \punt \bar{\alpha}_{\scriptscriptstyle D}.$
\end{proof}
Proposition~\ref{def2} gives (\ref{def:free cum gen fun}), if
we set $f(\bar{\alpha},t)= M(t),$ $f(\bar{\mathfrak{K}}_{\scriptscriptstyle
\alpha},t) = R(t)$ and observe that $f(\bar{\alpha}_{\scriptscriptstyle D},t) =
1 + t f(\bar{\alpha},t).$
\begin{thm}\label{thm f}
If ${\mathfrak{K}}_{\scriptscriptstyle \alpha}$ is the free
cumulant umbra of $\alpha$, then
${\bar{\mathfrak{K}}}_{\scriptscriptstyle \alpha} \equiv \bar{\alpha} \punt
\beta \punt {\bar{\alpha}_{\scriptscriptstyle D}}^
{{\scriptscriptstyle <-1>}}$ and $\bar{\alpha}
\equiv {\bar{\mathfrak{K}}}_{\scriptscriptstyle \alpha} \punt \beta \punt
{(-1 \punt {\bar{\mathfrak {K}}}_{\scriptscriptstyle \alpha})_{\scriptscriptstyle
D}^{\scriptscriptstyle <-1>}}.$
\end{thm}
\begin{proof}
The former similarity follows from Theorem~\ref{def2} as we have
$ \bar{\alpha} \punt \beta \punt {\bar{\alpha}_{\scriptscriptstyle D}}^
{{\scriptscriptstyle <-1>}} \equiv {\bar{\mathfrak{K}}}_
{\scriptscriptstyle \alpha} \punt \beta \punt \bar{\alpha}_
{\scriptscriptstyle D} \punt \beta \punt {\bar{\alpha}_{\scriptscriptstyle D}}^
{{\scriptscriptstyle <-1>}}$ and $\bar{\alpha}_{\scriptscriptstyle D} \punt \beta \punt
{\bar{\alpha}_{\scriptscriptstyle D}}^{{\scriptscriptstyle <-1>}} \equiv
\chi.$ The latter similarity follows from Definition \ref{free},
by observing that $\beta \punt
\bar{\alpha}_ {\scriptscriptstyle D} \equiv \beta \punt (-1  \punt
{\bar{\mathfrak{K}}_{\scriptscriptstyle \alpha})_{\scriptscriptstyle
D}^{\scriptscriptstyle <-1>}}.$
\end{proof}
A parametrization of free cumulants and moments can be constructed by
using the so-called \textit{umbral Abel polynomials}~\cite{DiNaSen3}
\begin{equation}\label{Abel}
A_n(x, \alpha) \simeq \left\{ \begin{array}{ll}
u & \hbox{if $n=0,$} \\
 x ( x - n \punt \alpha)^{n-1} & \hbox{if $ n \geq 1$}. \end{array} \right.
\end{equation}
Note that if the umbra $\alpha$ is replaced by the umbra $a \punt u,$ with
$u$ the unity umbra and $a \in R,$ then $E[A_n(x, a.u)] = A_n(x,a)$ for
all $n \geq 1,$ where $\{A_n(x,a)\}$ denotes the Abel polynomial sequence,
$A_n(x,a)=x(x-na)^{n-1}.$
\begin{thm}[Free parametrization] \label{thm:par free}
If ${\mathfrak{K}}_{\scriptscriptstyle \alpha}$ is the free
cumulant umbra of $\alpha$, then
\begin{equation}
\bar{\alpha}^n \simeq {\bar{\mathfrak{K}}}_{\scriptscriptstyle \alpha}
({\bar{\mathfrak{K}}}_{\scriptscriptstyle \alpha} + n.{\bar{\mathfrak{K}}}_
{\scriptscriptstyle \alpha})^{n-1} \quad \hbox{and} \quad
{\bar{\mathfrak{K}}}_{\scriptscriptstyle \alpha}^n \simeq \bar{\alpha}
(\bar{\alpha} - n.\bar{\alpha})^{n-1}. \label{(mom)}
\end{equation}
\end{thm}
\begin{proof}
In \cite{DiNaSen3}, the following equivalence $A_n(x, \alpha) \simeq (x \punt
\beta \punt \alpha_{\scriptscriptstyle D}^{\scriptscriptstyle <-1>})^n,$ is proved
for all $n \geq 1,$ so that $A_n({\bar{\mathfrak{K}}}_{\scriptscriptstyle \alpha},- 1 \punt
{\bar{\mathfrak{K}}}_{\scriptscriptstyle \alpha}) \simeq
[{\bar{\mathfrak{K}}}_{\scriptscriptstyle \alpha} \punt \beta \punt (- 1 \punt
{\bar{\mathfrak{K}}}_{\scriptscriptstyle \alpha})_{\scriptscriptstyle D}^{\scriptscriptstyle <-1>}]^n.$
From the latter similarity in Theorem~\ref{thm f}, we have $\bar{\alpha}^n
\simeq  A_n({\bar{\mathfrak{K}}}_{\scriptscriptstyle \alpha},- 1 \punt
{\bar{\mathfrak{K}}}_{\scriptscriptstyle \alpha}) \simeq
{\bar{\mathfrak{K}}}_{\scriptscriptstyle \alpha} (
{\bar{\mathfrak{K}}}_{\scriptscriptstyle \alpha}- n \punt (-1 \punt
{\bar{\mathfrak{K}}}_{\scriptscriptstyle \alpha}))^{n-1}$ by which the
former equivalence (\ref{(mom)}) follows. From the latter
similarity of Theorem~\ref{thm f}, we have ${\bar{\mathfrak{K}}}_{\scriptscriptstyle \alpha}^n
\simeq (\bar{\alpha} \punt \beta \punt {\bar{\alpha}_{\scriptscriptstyle D}^{\scriptscriptstyle <-1>}})^n
\simeq A_n(\bar{\alpha}, \bar{\alpha}).$ The latter equivalence (\ref{(mom)}) follows by
replacing $x$ with $\bar{\alpha}$ in (\ref{Abel}).
\end{proof}
\begin{cor} \label{cor1}
With $\{ r_n \}_{n \geq 1}$ and $\{ m_n \}_ {n \geq 1}$ given in (\ref{def:free cum
gen fun}), we have $ m_n = \sum_{\lambda\vdash n} (n)_{\ell(\lambda)-1} r_{\lambda} / m(\lambda)!$
and $r_n = \sum_{\lambda\vdash n}  (-n)_{\ell(\lambda)-1} m_{\lambda}/m(\lambda)!.$
\end{cor}
The Abel parametrization allows us to prove the homogeneity property of the
free cumulant umbra, since for any $c \in R$ and for any $\alpha \in A$
we have $-n \punt (c \alpha) \equiv c (- n \punt \alpha),$ see~\cite{DiNaSen1}.
\begin{thm} [Homogeneity property] \label{thm:hom}
If $\K_{\scriptscriptstyle\alpha}$ is the free cumulant
umbra of $\alpha$, then we have $\K_{\scriptscriptstyle c \alpha}\equiv c\K_{\scriptscriptstyle\alpha},$
for all $c\in R$.
\end{thm}
In order to prove the additivity property of the free cumulant umbra we
introduce an umbra $\delta_{\scriptscriptstyle P}$ such that
$(\delta_{\scriptscriptstyle P})^n\simeq \delta^{n+1} / (n+1)$
for $n=1,2,\ldots.$ Thanks to this device, Definition~\ref{free}
gives
$\bar{\K}_{\scriptscriptstyle \alpha}\equiv-1 \punt {(\ab)}_{\scriptscriptstyle P}.$
Denote by $\LL_{\bar{\alpha}}$ the umbra ${(\ab)}_{\scriptscriptstyle P}.$

Consider the multiplicative function $\mathfrak{f}$ on the noncrossing partition lattice defined by $\alpha^{n-1}\simeq f_n$. Note that $\mathfrak{f}$ is unital, that is $f_1=1$. The generating function $f(\LL_{\bar{\alpha}},t)$ is exactly the Fourier transform $(\mathcal{F}\mathfrak{f})(t)$ considered by Nica and Spei\-cher~\cite{NicSpe}. In particular, being $[\mathcal{F}(\mathfrak{f}\ast \mathfrak{g})](t)=(\mathcal{F}\mathfrak{f})(t)(\mathcal{F}\mathfrak{g})(t)$ for all $\mathfrak{f}$ and $\mathfrak{g}$ unital, if $\gamma^{n-1}\simeq g_n$ and $\omega^{n-1}\simeq h_n$, then we obtain $\mathfrak{h}=\mathfrak{f}\ast \mathfrak{g}\Leftrightarrow\LL_{\bar{\omega}}\equiv\LL_{\bar{\alpha}}+\LL_{\bar{\gamma}}$. This way, an analog of Theorem~\ref{thm fund 1} and Theorem~\ref{thm fund 2} for unital multiplicative functions on the noncrossing partitions lattice is given.
\begin{thm} [Additivity property] \label{thm:add free}
If $\K_{\scriptscriptstyle\alpha}, \K_{\scriptscriptstyle\gamma} $
and $\K_{\scriptscriptstyle\xi}$ are the free cumulant
umbrae of $\alpha, \gamma$ and $\xi$ respectively, then
\begin{equation}\label{free conv}
\K_{\scriptscriptstyle\xi} \equiv \K_{\scriptscriptstyle\alpha}
\stackrel{\punt}{+} \K_{\scriptscriptstyle\gamma} \Leftrightarrow
-1 \punt \LL_{\bar{\xi}} \equiv -1 \punt \LL_{\bar{\alpha}} \stackrel{\punt}{+}
-1 \punt \LL_{\bar{\gamma}}.
\end{equation}
\end{thm}
\begin{rem} [Connection between boolean and free convolution]
Write $\bar{\alpha}\dip$ for $(\ab)_{\scriptscriptstyle P}$. By
virtue of Theorem~\ref{thm:add free}, the \textit{free
convolution} $\alpha{\free}\gamma$ of $\alpha$ and $\gamma$ has to
be defined by
$\overline{\alpha{\free}\gamma}\dip\equiv-1\punt[-1\punt\bar{\alpha}\dip\stackrel{\punt}{+}-1\punt\bar{\gamma}\dip],$
so that
$\K_{\scriptscriptstyle\alpha{\free}\gamma}\equiv\K_{\scriptscriptstyle\alpha}{\stackrel{\punt}{+}}
\K_{\scriptscriptstyle\gamma}.$ Moreover, thanks to the umbra
$\LL_{\bar{\alpha}}$ we have
$$\LL_{\scriptscriptstyle\overline{\alpha{\free}\gamma}}\equiv\LL_{\scriptscriptstyle\bar{\alpha}}\uplus\LL_{\scriptscriptstyle\bar{\gamma}},$$
which gives the connection between boolean and free convolution.
\end{rem}
Semi-invariance property can be proved by observing that
$\K_{\scriptscriptstyle\alpha\free
cu}\equiv\K_{\scriptscriptstyle\alpha}
{\stackrel{\punt}{+}}c\K_{\scriptscriptstyle u}$ so that
$\K_{\scriptscriptstyle\alpha\boxplus c
u}\equiv\K_{\scriptscriptstyle\alpha} {\stackrel{\punt}{+}}c
\chi,$ being $\bar{\K}_{\scriptscriptstyle u}\equiv\bar{u}\punt
\beta \punt {\bar{u}_{\scriptscriptstyle D}}^{{\scriptscriptstyle
<-1>}}\equiv\chi$.

\begin{defn} [Catalan umbra] The {\it Catalan umbra} is the unique
umbra $\varsigma$ such that $\K_{\varsigma}\equiv u$,
that is $\bar{\varsigma}\equiv \bar{u} \punt \beta \punt (-1 \punt
\bar{u})_{\scriptscriptstyle D}^{\scriptscriptstyle <-1>}.$
\end{defn}
As it is well known, Catalan numbers count the noncrossing
partitions of a set. So in the free setting, the Catalan umbra plays the same role
played by the Bell umbra $\beta$ in the classical framework.
\begin{prop}[Catalan numbers]
If $\mathcal{C}_n$ is the $n$-th Catalan number, then
$\varsigma^n\simeq \mathcal{C}_n.$
\end{prop}
\begin{proof}
We have $n!\varsigma^n\simeq\bar{\varsigma}^n\simeq n!\sum_{\mu\vdash
n}(n)_{\ell(\mu)-1}/m(\mu)!. $
As well known (see for instance~\cite{Krew}),
$(n)_{\ell(\mu)-1}/m(\mu)!$ is the number of noncrossing
partitions of shape $\mu$  and $|\mathcal{NC}_n|=\mathcal{C}_n$, so
that $\varsigma^n\simeq |\mathcal{NC}_n|=\mathcal{C}_n.$
\end{proof}
%
\section{Volume polynomial}
%
In this section we provide an explicit connection between free
cumulants and parking functions via volume polynomials. Moreover
we prove that in the free setting the volume polynomials play
the same role played by the complete Bell exponential polynomials
in the classical settings.
\par
Recall that a \emph{parking function} of length $n$ is a sequence
$(p_1,p_2,\ldots,p_n)$ of $n$ positive integers, whose
nondecreasing arrangement $(p_{i_1},p_{i_2},\ldots,p_{i_n})$ is such that
$p_{i_j}\leq j$. We denote by $\park(n)$ the set of all parking
functions of length $n$; its cardinality is $(n+1)^{n-1}$. The
symmetric group $\frak{S}_n$ acts on the set $\park(n)$ by
permuting the entries of parking functions. As well known, the
number of orbits in $\park(n)^{\frak{S}_n}$ is equal to the
$n$-th Catalan number $\mathcal{C}_n$. It is also known that a map
$\tau$ can be defined from $\park(n)$ to $\mathcal{NC}_n$ whose
restriction to $\park(n)^{\frak{S}_n}$ is bijective. The $n$-volume
polynomial $V_n(x_1,x_2,\ldots,x_n)$, introduced by Pitman and Stanley~\cite{PitStan}, is the following
homogeneous polynomial of degree $n$:
\begin{equation}\label{def:vol pol 1}
V_n(x_1,x_2,\ldots,x_n)=\frac{1}{n!}\sum_{p\in \park(n)}\!\!\!\!x_p,
\end{equation}
where $x_p=x_{p_1}x_{p_2}\cdots x_{p_n}$ whenever $p=(p_1,p_2, \ldots,p_n)$. For each
$p \in \park(n)$ let
$m(p)=(m_1,m_2,\ldots,m_n)$ be the vector of the multiplicities of
$p$, that is $m_j=|\{i\,|\,p_i=j\}|.$
If $\lambda$ is a partition of $n$, then we say that the parking
function $p$ is of type $\lambda$ if the nonzero entries of $m(p)$
consists of a rearrangement of the parts of $\lambda$. The orbit
$\mathcal{O}_p= \{\omega(p)\,|\,\omega\in\frak{S}_n\}$ of a
parking function of type $\lambda$ has cardinality $n!/\lambda!$.
The map $\tau$ has the following property: $p$ is of type
$\lambda$ if and only if $\tau(p)=\{B_1,B_2,\ldots,B_l\}$ is of shape $\lambda$. Hence, the polynomial $V_n(x_1,x_2,\ldots,x_n)$ can be written as
\begin{equation}\label{def:vol pol 2}
V_n(x_1,x_2,\ldots,x_n)=\sum_{\lambda\vdash
n}\frac{1}{\lambda!}\frac{(n)_{\ell(\lambda)-1}}{m(\lambda)!}x^
{\lambda},
\end{equation}
being $x^{\lambda}=x_1^{\lambda_1}x_2^{\lambda_2}\cdots
x_l^{\lambda_l}$. In particular when $x_i$ are replaced
by similar and uncorrelated umbrae we have $n!V_n(\alpha',\alpha'',\ldots,\alpha''')\simeq\alpha(\alpha+n\punt\alpha)^{n-
1},$ for all $\alpha\in A$ (see~\cite{PetSen}). By using this last
result and Theorem~\ref{thm:par free}, the following theorem provides
an explicit connection between free cumulants and parking functions.
\begin{thm}
Let $\alpha$ be an umbra and let $\K_{\scriptscriptstyle\bar{\alpha}}$
be its free cumulant umbra. If $\K'$, $\K''$,\ldots, $\K'''$
are $n$ uncorrelated umbrae similar to
$\K_{\scriptscriptstyle\alpha}$ and $V_n(x_1,x_2,\ldots,x_n)$ is the
$n$-volume polynomial (\ref{def:vol pol 1}), then
$\bar{\alpha}^n\simeq V_n(\bar{\K}',\bar{\K}',\ldots,\bar{\K}''').$
\end{thm}
\begin{cor}
If $\varsigma$ is the Catalan umbra and $u',u'',\ldots,u'''$ are
uncorrelated umbrae similar to the unity $u$, then $\bar{\varsigma}^n\simeq V_n(\bar{u}',\bar{u}'',\ldots,\bar{u}'''),$
or equivalently $n!\mathcal{C}_n = E[\bar{u}(\bar{u}+n\punt\bar{u})^{n-1}].$
\end{cor}
Observe that, from \eqref{def:vol pol 1} we have $n!V_n(x_1,x_2,\ldots,x_n)=\sum_{p\in\park(n)}x_p$. If we restrict the sum to the quotient $\park(n)^{\frak{S}_n}$ (i.e. if we take only a parking function per orbit) we obtain polynomials $R_n(x_1,x_2,\ldots,x_n)=\sum_{\lambda\vdash n}(n)_{\ell(\lambda)-1}x_p/m(\lambda)!$ such that $R_n(\K',\K'',\ldots,\K''')\simeq m_n$. Thanks to the parametrization given in Theorem~\ref{par1} and Theorem~\ref{par2} we can also construct polynomials $C_n(x_1,x_2,\ldots,x_n)$ and
$H_n(x_1,x_2,\ldots,x_n)$ such that $m_n = C_n(\kappa',\kappa'',\ldots,\kappa''')=H_n(\eta',\eta'',\ldots,\eta''')$, $\kappa',\kappa'',\ldots,\kappa'''$ and $\eta',\eta'',\ldots,\eta'''$ being
uncorrelated umbrae similar to $\kappa_{\scriptscriptstyle\alpha}$ and $\eta_{\scriptscriptstyle\alpha}$ respectively. This will be done in the next section for a more general class of cumulants.

Finally, since we have $E[C_n(\kappa',\kappa'',\ldots,\kappa''')]=Y_n(c_1,c_2,\ldots,c_n)$, then
the analog of the complete Bell polynomials in the boolean and free case are the polynomials
$E[H_n(\eta',\eta'',\ldots,\eta''')]$ and $E[R_n(\K',\K'',\ldots,\K''')]$  respectively.
%
\section{Linear cumulants and Abel polynomials}
%
Let $\{g_n\}_{n \geq 1}$ be a sequence of nonnegative integers represented
by an umbra $\gamma.$ Let us define
the \textit{generalized Abel polynomials} as the umbral polynomials
$A_n^{(\gamma)}(\delta,\alpha)$ such that $A_n^{(\gamma)}(\delta,\alpha) \simeq
\delta ( \delta - g_n \punt \alpha)^{n-1}$ for $n \geq 1.$ In particular, when $\alpha\equiv\delta$
we will write $A_n^{(\gamma)}(\alpha)$ instead of $A_n^{(\gamma)}(\alpha,\alpha)$.
It can be shown that (see \cite{PetSen}, Theorem~3.1)
\begin{equation}\label{lemma0}
A_n^{(\gamma)}(\alpha) \simeq \sum_{\lambda\vdash n}\mathrm{d}_{\lambda}(-g_n)_{\ell(\lambda)-1}(\alpha')^{\lambda_1}(\alpha'')^{\lambda_2}\cdots(\alpha''')^{\lambda_{\ell(\lambda)}}.
\end{equation}
Generalized Abel polynomials allow us to express classical, boolean and free cumulants
in terms of moments. Indeed for the classical cumulants from Theorem~\ref{par1} we have $\kappa_{\scriptscriptstyle
\alpha}^n \simeq A_n^{(u)}(\alpha),$ since the sequence $\{1\}_{n \geq 1}$
is represented by the unity umbra $u.$ Since the sequence $\{2\}_{n \geq 1}$
is represented by the umbra $2\punt u,$ from Theorem~\ref{par2}
we have $ \bar{\eta}_{\alpha}^n \simeq A_n^{(2\punt u)}(\bar{\alpha})$
for the boolean cumulants. Since the
sequence $\{n\}_{n \geq 1}$ is represented by the umbra $u_{\scriptscriptstyle D}$,
from Theorem~\ref{thm:par free}  we have ${\bar{\mathfrak{K}}}_{\scriptscriptstyle
\alpha}^n \simeq A_n^{(u_{\scriptscriptstyle D})}(\bar{\alpha})$ for the free cumulants.
In this section, by using generalized Abel polynomials, we show how to construct a more general family of
cumulants possessing the additivity, homogeneity and semi-invariance properties. To the best of our knowledge,
a previous attempt to give a unifying approach to cumulants families was given in
\cite{Ans}, but the boolean case seems not fit in.
\begin{defn}\label{cu}[Cumulant umbrae] The umbra $\K_{\gamma,\alpha}$ such that
$\K_{\gamma,\alpha}^n \simeq A_n^{(\gamma)}(\alpha)$ for all $n \geq 1$
is called the \textit{cumulant umbra} of $\alpha$ induced by the umbra $\gamma.$
\end{defn}
Let ${\bf a}=(a_n)_{n \geq 1}$ and ${\bf g}=(g_n)_{n \geq 1}$ be the sequences of moments
of $\alpha$ and $\gamma$ respectively. Then the $n$-th \textit{cumulant} of $\alpha$ induced by $\gamma$
 is $c_{n}({\bf a};{\bf g}) = E[\K_{\gamma,\alpha}^n].$ If we choose as umbra $\gamma$ the umbra $k\punt u$
and we set $c_{n,k} = E[A_n^{\scriptscriptstyle(k\punt u)}(\alpha)]$, then we may consider the infinite matrix
$$C({\bf a})=\left[\begin{array}{ccc}
c_{1,1}&c_{1,2}&\cdots\\
c_{2,1}&c_{2,2}&\cdots\\
\vdots & \vdots&\ddots
\end{array} \right].$$
The cumulants induced by the umbra $k\punt u$ are the ones occurring in the $k$-th column. But we
can construct different sequences of cumulants of $\alpha$ by extracting one entry from each row of
$C({\bf a}).$ For example, suppose to define the umbra $\gamma$ such that $g_n=(n+k-1),$ for all
$n \geq 1$ and for a fixed positive integer $k.$ The cumulants induced by this umbra $\gamma$ are the
ones occurring in the $k$-th diagonal of $C({\bf a}).$
\par
By means of equivalence (\ref{lemma0}) and Definition \ref{cu}, we have
\begin{equation}
\label{Q}
\K_{\gamma,\alpha}^n \simeq Q_{n}(\gamma; \alpha',\alpha'',\ldots,\alpha''')
\end{equation}
where $Q_n(\gamma;x_1,x_2,\ldots,x_n) = \sum_{\lambda\vdash n}d_{\lambda} (-g_n)_{\ell(\lambda)-1}x_1^{\lambda_1}x_2^{\lambda_2}\cdots x_n^{\lambda_n}$ ($\lambda_i=0$ if $i>\ell(\lambda)$)
are homogeneous polynomial in $R[X]$ of degree $n$ whose coefficients do not depend on $\alpha.$
This property of $Q_{n}(\gamma; x_1,\ldots,x_n)$ gives rise to the following theorem.
\begin{thm}[Homogeneity property] \label{hp}
If $\K_{\gamma,\alpha}$ is the \textit{cumulant umbra} of $\alpha$ induced by the umbra $\gamma,$
then $\K_{\gamma,j\alpha}\equiv j\K_{\gamma,\alpha}$ for all $j\in R.$
\end{thm}
If we set ${\bf j \, a}=(j^n a_n)_{n \geq 1},$ then the homogeneity  property states that
$c_{n}({\bf j \, a};{\bf g}) = j^n c_{n}({\bf a};\bf{g})$ for all $n.$ In terms of the matrix
$C({\bf a})$, the homogeneity  property can be restated as
$C({\bf j  \, a})^T = \hbox{diag}(j,j^2, \cdots) C({\bf a})^T.$
It is also possible to express the moments of $\alpha$ in terms of its cumulants induced by any $\gamma$ with positive integer moments.
\begin{thm}[Invertibility property] \label{In}
For all scalar umbrae $\gamma$ whose moments $\{g_n\}_{n \geq 1}$ are positive integers, there exists a
sequence $\{P_{n}(\gamma;x_1,\ldots,x_n)\}_{n \geq 1}$ of homogeneous umbral polynomials of degree $n$, such that
for all $n$ and for all $\alpha \in A$ we have $\alpha^n \simeq P_{n}(\gamma;\K',\K'',\ldots,\K'''),$
for all $n$-sets $\{\K',\K'',\ldots,\K'''\}$ of umbrae similar to $\K_{\gamma,\alpha}.$
\end{thm}
\begin{proof}
Suppose to denote by $c_n$ the $n$-th moment of $\K_{\gamma,\alpha}.$
From (\ref{Q}), $c_n =  a^n + q(a_1,a_2,\ldots,a_{n-1})$ where $q$ is a suitable
polynomial in $a_1,a_2,\ldots,a_{n-1}.$  So $a_n$ can be expressed in terms of
$c_1, \ldots, c_n$ by recursions. By replacing occurrences
of product of powers of the $c_i$'s by suitable products of powers of the $x_i$'s, the
polynomials $P_n$ such that $a_n=E[P_n(\gamma;\K',\K'',\ldots,\K''')]$ can be
constructed from these expressions. Finally, from the homogeneity property~\ref{hp}, we have
$P_{n}(\gamma;j\K',j\K'',\ldots,j\K''') \simeq j^n \alpha^n$, which assures the homogeneity of the $P_n$'s.
\end{proof}
Each sequence of cumulants linearizes a certain convolution of umbrae (i.e. of moments) and this is
why we call the elements of the matrix $C({\bf a})$ linear cumulants. More precisely, we define the \textit{convolution} of two umbrae $\alpha$ and $\eta$ induced by the umbra $\gamma$ to be the auxiliary umbra $\alpha +_{\kcon} \eta$ such that
\begin{equation}
\label{add}
\K_{\gamma,\alpha +_{\kcon}\omega}\equiv \K_{\gamma,\alpha}\stackrel{\punt}{+} \K_{\gamma,\omega},\,\,\,\,\,\,\,\,\,\text{(Additivity property).}
\end{equation}
In particular, convolutions are commutative. The invertibility property~\ref{In} assures the existence of the convolution of any pair of umbrae induced by any
umbra whose moments are positive integers.
\begin{thm}
For all scalar umbrae $\gamma$ whose moments $\{g_n\}_{n \geq 1}$ are positive integers, there
exists a sequence of polynomials $T_n(\gamma;x_1,\ldots,x_n,
y_1,\ldots,y_n)$  homogeneous of degree $n$ for all $n$, such that
for all $n$ and for all scalar umbrae $\alpha, \omega \in A$ we have
$(\alpha +_{\kcon} \omega)^n\simeq T_n(\gamma;\alpha',\alpha'',\ldots,\alpha''',\omega',\omega'',\ldots,\omega''').$
\end{thm}
\begin{proof}
Due to the invertibility property~\ref{In}, there exists $P_n(\gamma;x_1,x_2,\ldots,x_n)$ such that
$(\alpha +_{\kcon} \omega)^n \simeq P_n(\gamma;\K',\K'',\ldots,\K''').$ Then, suppose to replace
each occurrence of $x_i^{\lambda_i}$ in $P_n$ with $x_i^{\lambda_i}+y_i^{\lambda_i}$ and denote
by $T_n$ the polynomial resulting of this replacement. By virtue of the additivity property~(\ref{add}), it is straightforward to prove that $T_n$ satisfies all the properties of the theorem.
\end{proof}
In general, the cumulant umbrae $\K_{\scriptscriptstyle \gamma,\alpha}$'s do not have the semi-invariance property. This is due to the fact that $\K_{\gamma,u}$ is not similar to $\chi$, so that $\K_{\scriptscriptstyle \gamma,\alpha+_{\kcon}c u}$ is not similar to $\K_{\scriptscriptstyle \gamma,\alpha}\stackrel{\punt}{+}c\chi.$ However, after a suitable normalization of cumulants, moments and convolutions it is possible to recover the semi-invariance property. More explicitly, for the first column (classical cumulants) no normalization is needed. For the second column the right normalization (which returns boolean cumulants) is obtained via the moments $n!$ of the boolean unity $\bar{u}$. Indeed, $\{\K_{\scriptscriptstyle 2\punt u,\alpha}^n/n!\}_{n\geq 1}$ is a sequence of cumulants for the moments $\{\alpha^n/n!\}_{n\geq 1}$ which is semi-invariant with respect to the convolution $\{(\alpha+_{\scriptscriptstyle(2\punt u)}\omega)^n/n!\}_{n\geq 1}$. For the main diagonal (free cumulants) it is again $\bar{u}$ giving a good normalization. More generally, for columns and diagonals the normalization is always possible and it is obtained via umbrae representing positive integer moments.
\end{document}